\documentclass[10pt]{amsart}

\usepackage{mathrsfs}
\usepackage{enumerate}
\newtheorem{theorem}{Theorem}
\newtheorem{lemma}[theorem]{Lemma}

\newtheorem{corollary}[theorem]{Corollary}

\theoremstyle{definition}

\newtheorem{example}[theorem]{Example}
\newtheorem{remark}[theorem]{Remark}
\allowdisplaybreaks

\def\NN{\mathbb{N}}
\def\CC{\mathbb{C}}

\def\ZZ{\mathbb{Z}}

\def\id{I}
\def\dd{\mathrm{d}}

\def\RR{\mathbb{R}}
\def\cc{\mathrm{c}}
\def\co{\mathrm{c}_0}
\def\BUC{\mathrm{BUC}}
\def\ran{\mathop{\mathrm{ran}}}

\def\rg{\mathop{\mathrm{rg}}}

\def\one{\mathbf{1}}

\def\LLL{\mathscr{L}}

\def\veps{\varepsilon}

\def\aap{{\text{\rm aap}}}
\def\st{{\text{\rm s}}}
\def\rv{{\text{\rm r}}}

\begin{document}
\title[A Bohl--Bohr--Kadets type theorem]{A Bohl--Bohr--Kadets type theorem characterizing Banach spaces not containing $c_0$}
\author{B\'alint Farkas}
%\address{E\"otv\"os Lor\'and University\newline Institute of Mathematics\newline Department of Applied Analysis and Computational Mathematics\newline P\'azm\'any P\'eter s\'et\'any 1/C, 1117, Budapest, Hungary}
%\email{fbalint@cs.elte.hu}
\address{Bergische Universit\"at Wuppertal\newline Faculty of Mathematics and Natural Science\newline Gaussstrasse 20, D-42119 Wupertal, Germany}
\email{farkas@uni-wuppertal.de}
\begin{abstract}
We prove that a separable Banach space $E$ does not contain a copy of the space $\co$ of null-sequences if and only if for every doubly power-bounded operator $T$ on $E$ and for every vector $x\in E$ the relative compactness of the sets $\{T^{n+m}x-T^nx: n\in \NN\}$ (for some/all $m\in\NN$, $m\geq 1$) and $\{T^nx:n\in \NN\}$ are equivalent. With the help of the Jacobs--de Leeuw--Glicksberg decomposition of strongly compact semigroups the case of (not necessarily invertible) power-bounded operators is also handled.
\end{abstract}
\thanks{This paper was supported by the J\'anos Bolyai
Research Fellowship of the Hungarian Academy of Sciences and
by the Hungarian Research Fund (OTKA-100461).}
\keywords{Power-bounded linear operator, (asymptotically) almost periodic vectors and differences, the Banach space of null-sequences, Kadets' theorem}
\subjclass[2010]{47A99, 46B04, 43A60}
 \maketitle

This note concerns the following problem: Given a Banach space $E$, a bounded linear operator $T\in \LLL(E)$ and a vector $x\in E$, we would like to
conclude the relative compactness of the orbit
\begin{align*}
&\bigl\{T^nx:\:n\in\NN\bigr\}\subseteq E
\intertext{from the relative compactness of the consecutive differences of the iterates}
&\bigl\{T^{n+1}x-T^nx:\:n\in\NN\bigr\}\subseteq E.
\end{align*}
This problem is a discrete, ``linear operator analogue'' of the classical Bohl--Bohr theorem about the integration of almost periodic functions. Before going to the results let us explain this connection.

\medskip\noindent Given a (Bohr) almost periodic function $f:\RR\to \CC$ with its integral $F(t)=\int_0^tf(s)\dd s$ bounded, then $F$ is almost periodic itself. This result was extended to Banach space valued almost periodic functions $f:\RR\to E$ by M.~I. Kadets \cite{KadecMI}, provided that $E$ does not contain an isomorphic copy of the Banach space $\co$ of null-sequences. Actually, the validity of this integration result for \emph{every} almost periodic function $f:\RR\to E$ characterizes the absence of $\co$ in the Banach space $E$.

\medskip\noindent The generalization of Kadets' result---which explains the connection to our problem---was studied by Basit for functions $f:G\to E$ defined on a group $G$ and taking values in the Banach space $E$ (for simplicity suppose now $G$ to be Abelian).  In \cite{Bolis71} Basit proved that if  $F:G\to E$ is a bounded function with almost periodic difference functions
$$
F(\cdot +g)-F(\cdot)\quad\mbox{for all $g\in G$},
$$
and $E$ does not contain $\co$, then $F$ is almost periodic. The relation to Kadets' result is the following: If $f:\RR\to E$ is almost periodic, so is $F_\veps(t):=\int_{t}^{t+\veps} f(s)\dd s$ for every $\veps>0$. So Kadets' theorem tells that if $F(t)=\int_0^t f(s)\dd s$ is bounded and $E$ does not contain $\co$, then the almost periodicity of
$$
F(\cdot+\veps)-F(\cdot)\quad\mbox{for all $\veps>0$}
$$
implies that of $F$.

\medskip\noindent Now returning to our problem, suppose $T\in \LLL(E)$ is \emph{doubly power-bounded} (i.e., $T,T^{-1}\in \LLL(E)$ are both power-bounded). Then by applying Basit's result to the function   $F:\ZZ\to E$, $F(n):=T^nx$, we obtain that $\{T^nx:n\in \ZZ\}$ is relatively compact in $E$ if
$$
\bigl\{T^{n+m}x-T^nx:n\in \ZZ\bigr\}\quad\mbox{is relatively compact for all $m\in \ZZ$},
$$
for which it suffices that
$$
\bigl\{T^{n+1}x-T^nx:n\in \ZZ\bigr\}\quad\mbox{is relatively compact}.
$$
Let us record this latter fact in the next lemma.
\begin{lemma}\label{lem:A1Am}
Let $E$ be a Banach space and $T\in\LLL(E)$ be a power-bounded operator.
 If for some $x\in E$ the set
$$
 \bigl\{T^{n+1}x-T^nx:n\in \NN\bigr\}
 $$
 is relatively compact, then so is the set
  $$
 \bigl\{T^{n+m}x-T^n x:n\in \NN\bigr\}
 $$
for all $m\in \NN$.
 \end{lemma}
 \begin{proof}
Denote by $D_1$ the first set and by $D_m$ the second one. We may suppose $m\geq 2$.  By continuity of $T$ the sets $TD_1$, $T^2D_1$, $T^{m-1}D_1$ are all relatively compact. Since
$$
T^{m+n}x-T^nx=
T^{m-1}(T^nx-x)+T^{m-2}(T^nx-x)+\cdots +T(T^nx-x)+(T^n x-x),
$$
we obtain $D_m\subseteq TD_1+T^2D_1+\cdots +T^{m-1}D_1$ implying the relative compactness of $D_m$.
 \end{proof}

So our problem can be answered satisfactorily for doubly power-bounded operators on Banach spaces not containing $\co$. The situation is different if $T$ is non-invertible, or invertible but with not power-bounded inverse. To enlighten what may be true in such a situation, let us recall a result of Ruess and Summers, who considered the generalization of the Bohl--Bohr--Kadets result for functions $f:\RR_+\to E$,  see \cite[Thm.~2.2.2]{RuSu1} or \cite[Thm.~4.3]{RuSu2}.

\medskip\noindent  Given an \emph{asymptotically almost periodic} function $f:\RR_+\to E$, one can find an \emph{almost periodic} one $f_\rv:\RR\to E$ and another function $f_\st:\RR_+\to E$ vanishing at infinity such that $f=f_\st+f_\rv$. Ruess and Summers proved the following. Suppose  $f$ is asymptotically almost periodic with
$$
F(t):=\int_0^t f(s)\dd s\quad\mbox{bounded},
$$
and the improper Riemann integral of $f_\st$ exists in $E$. If $E$ does not contain $\co$ then $F$ is asymptotically almost periodic. For details and discussion we refer to \cite[Sec.~2.2]{RuSu1}. As we see, a Jacobs--de Leeuw--Glicksberg type decomposition plays an essential role here.

\medskip\noindent Our main result, Corollary \ref{cormain}, provides the solution to the very first problem concerning power-bounded operators in this spirit. It contains the mentioned special case of Bolis' result when $T$ is doubly power-bounded. For stating the result we first need some preparations,  explaining an analogue of the decomposition above used by Ruess and Summers. Let $E$ be a Banach space and let $T\in \LLL(E)$, which is from now on always assumed to be power-bounded. A vector $x\in E$ is called \emph{asymptotically almost periodic} (a.a.p.~for short) with respect to $T$ if the (forward) orbit
$$
\bigl\{T^n x:n\in \NN\bigr\}\subseteq E
$$
is relatively compact. Denote by $E_{\aap}$ the collection of a.a.p.~vectors, which is a closed $T$-invariant subspace of $E$. We shall need the following form of the Jacobs--de Leeuw--Glicksberg decomposition for operators with relatively compact (forward) orbits; see \cite[Chapter 16]{EFHN}, or \cite[Thm~V.2.14]{EN} where the proof is explained for continuously parametrized semigroups instead of semigroups of the form $\{T^n:n\in \NN\}$ (the proof is nevertheless the same).

\begin{theorem}[Jacobs--de Leeuw--Glicksberg]
Let $E$ be a Banach space and let $T\in \LLL(E)$ have relatively compact orbits ($T$ is hence power-bounded). Then there is a projection $P\in \LLL(E)$ commuting with $T$ such that
\begin{align*}
E_\rv:=\rg P&=\bigl\{x\in E: Tx=\lambda x\mbox{ for some } \lambda\in \CC,\:|\lambda|=1\bigr\},\\
E_\st:=\rg(I-P)=\ker P&=\bigl\{x\in E:T^nx\to 0\mbox{ for $n\to\infty$}\bigr\}.
\end{align*}
The restriction of $T$ to $E_\rv$ is a doubly power-bounded operator.
\end{theorem}
Note that the occurring subspaces and the projection depend on the linear operator $T$  and---for the sake of better legibility, we chose not to reflect this dependence in notation. Now, we can apply this decomposition to a  given power-bounded $T\in \LLL(E)$, or more precisely to the restriction of $T$ to $E_\aap$. We therefore obtain a decomposition
$$
E_{\aap}=\rg P\oplus \ker P=E_\rv\oplus E_\st,
$$
note that $E_\rv,E_\st\subseteq E_{\aap}$; $E_\st$ is called the stable while $E_\rv$ the reversible subspace.
On $\rg P$ the restriction of $T$ is a doubly power-bounded, and
$$
\bigl\{T^n|_{E_\rv}:n\in \ZZ\bigr\}
$$
is a strongly compact group of operators.
\begin{remark}\label{rem:dbly}
Now suppose that $T\in \LLL(E)$ is even doubly power-bounded Then in the above decomposition $E_\st=\{0\}$ must hold. So if $x\in E_\aap$, then even the backward orbit is relatively compact, i.e.
$$
\bigl\{T^nx:n\in \ZZ\bigr\} \quad\mbox{is relatively compact}.
$$
A vector $x\in E_\rv$ is also called \emph{almost periodic}.
\end{remark}
Lemma \ref{lem:A1Am} tells that for $n,m\in \NN$ with $m\geq n$ we have  $(T^m-T^n)x\in E_\aap$ whenever $(T-I)x\in E_\aap$.

\medskip\noindent We are interested in  whether $x\in E_\aap$ if $(T-I)x\in E_\aap$. The answer would be trivially ``yes'', if we knew that $T^nx-x$ actually converges as $n\to\infty$. Here is a slightly more complicated view on trivial fact:
\begin{remark}
\begin{enumerate}[a)]
\item Suppose we know $x\in E_\aap$. Then we can apply $I-P$ to $x$ and obtain
\begin{equation*}
(I-P)(T^n-I)x=(I-P)(T^nx-x)=T^n(I-P)x-(I-P)x\to (P-I)x
\end{equation*}
for $n\to\infty$.
Hence if $x\in E_\aap$, then $(I-P)(T^n-I)x$ must be convergent.
\item
Suppose that $(I-P)(T^n-I)x$ converges for $n\to\infty$ (note again that $T^nx-x\in E_\aap$, so we can apply the projection $I-P$ to it). If $(T-I)x\in E_\aap$ belongs even to the stable part, then $(T^n-I)x\in E_\st$, so $(I-P)(T^n-I)x=(T^n-I)x$. Hence $T^nx-x$ converges as $n\to \infty$ by assumption, implying $x\in E_\aap$.
\end{enumerate}
\end{remark}

\noindent It remains to study the case when $(T-I)x\in E_\rv=\rg P$. The next is  a preparatory lemma.
\begin{lemma}\label{lem:Pbig}
Suppose $x$ is not an a.a.p.~vector but $(T-I)x\in E_\aap$. Furthermore, suppose that $(I-P)(T^nx-x)$ converges. Then there is a $\delta>0$ and a subsequence $(n_k)$ of $\NN$ such that
$$
\|P(T^{n_k}x-T^{n_\ell}x)\|\geq \delta\quad\mbox{for all $k,\ell\in\NN$, $k\neq\ell$.}
$$
\end{lemma}
\begin{proof}
By the non-a.a.p.~assumption there is a subsequence $(n_k)$ of $\NN$ and a $\delta>0$ such that $\|T^{n_k}x-T^{n_\ell}x\|>2\delta$ for $\ell\neq k$. By the other assumption, however,
$(I-P)(T^{n_k}x-x)$ is a Cauchy-sequence so when leaving out finitely many members, we can pass to a subsequence with $\|(I-P)(T^{n_k}x-T^{n_\ell}x)\|<\delta$ for all $k,\ell\in \NN$, $\ell,k\geq k_0$. The assertion follows from this.
\end{proof}
Note that in the situation of this lemma we necessarily have $\|P\|>0$.

\begin{lemma}\label{lem:multap}
Let $E$ be a Banach space,  let $T\in \LLL(E)$ be power-bounded, and let $x_1,\dots x_m\in E$ be a.a.p.~vectors.
For every sequence $(n_k)\subseteq\NN$  there is a subsequence $(n'_k)$ with $n'_k-n'_{k-1}\to \infty$ and
$$\|T^{n'_{k}}x_i-T^{n'_{k-1}}x_i\|\to 0\quad \mbox{ for all $i=1,\dots, m$ as $k\to\infty$}.
$$
\end{lemma}
\begin{proof}
Consider the Banach space $X=E^m$ and the diagonal operator $S\in \LLL(X)$ defined by $S(y_i)=(Ty_i)$. This is trivially power-bounded and $(x_i)\in X$ is an a.a.p.~vector. The assertion follows from this, since $(S^{n_k}(x_i))$ has a Cauchy subsequence $(S^{n'_k}(x_i))$ with $n'_k-n'_{k-1}\to \infty$ as  $k\to\infty$.
\end{proof}

%\begin{lemma}\label{lem:multap}Let $E$ be a Banach space,  let $T\in \LLL(E)$ be doubly power-bounded and let $x_1,\dots x_m\in E$ be a.a.p.~vectors. There is subsequence $(n_k)$ (even a relatively dense one) such that $\|T^{n_k}x_i-x_i\|\to 0$ as $k\to\infty$ for all $i=1,\dots, m$.
%\end{lemma}
%\begin{proof}
%Consider the Banach space $X=E^m$ and the diagonal operator $S\in \LLL(X)$ defined by $S(y_i)=(Ty_i)$. This is trivially doubly power-bounded and $(x_i)\in X$ is an almost periodic vector. The assertion follows from this, because a.a.p.~vectors are recurrent.
%\end{proof}

We now come to the answer of the initial question.
\begin{theorem}\label{thm:bolisnoco1}
Let $E$ be Banach space which
does not contain an isomorphic copy of $\co$, and let $T\in \LLL(E)$ be a power-bounded operator. If $x\in E$ and $(T-I)x$ is an a.a.p.~vector with $(I-P)(T^n-I)x$ convergent for $n\to \infty$, then  $x$ itself is a.a.p.~vector.
\end{theorem}
The proof is by contradiction, i.e., we suppose that there is some $x\in E$ satisfying the assumptions of the theorem but being not asymptotically almost periodic. The contradiction arises then by finding a copy of $\co$ in $E$, for which we shall use the classical result of Bessaga and Pe{\l}czy\'nski \cite{BessPel} in the following form, see also \cite[Thms.~6 and 8]{Di84}.
\begin{theorem}[Bessaga--Pe{\l}czy\'nski]\label{thm:bespelcz} Let $E$ be a Banach space and let $x_n\in E$ be vectors such that the partial sums are unconditionally bounded (i.e., $\sum_{j=1}^N x_{n_j}$ are uniformly  bounded for all subsequences $(n_j)$ of $\NN$) and such that the series $\sum x_i$ is nonconvergent. Then $E$ contains a copy of $\co$.
\end{theorem}
 The idea of the proof is based on Basit's paper, but it is not a direct modification, since we do not know whether we can apply the projections $P$ to $x$ or to $Tx$.
\begin{proof}[Proof of Theorem \ref{thm:bolisnoco1}]
We argue indirectly. Assume that $x\not\in E_{\aap}$, so by Lemma \ref{lem:Pbig} we can take a subsequence $(n_k)$ and a $\delta>0$ such that $\|P(T^{n_k}x-T^{n_\ell}x)\|>\delta$ for all $k,\ell\in\NN$ with $k\neq \ell$.
Next we construct a sequence that fulfills the conditions of the Bessaga--Pelczynski Theorem \ref{thm:bespelcz}, hence exhibiting a copy of $\co$ in $E$. First of all let $M:=\max(\sup\{\|T^n|_{E_\rv}\|:n\in \ZZ\},\sup\{\|T^n\|:n\in \NN\},\|P\|)$.  Take $k_1\in \NN$ such that $\|P(T^{{k_1}}x-x)\|>\delta/M$ (use Lemma \ref{lem:Pbig}), and suppose that the strictly increasing finite sequence $k_i$, $i=1,\dots,m$ is already chosen. For a subset $F\subset \{1,2,\dots,m\}$ denote by $\Sigma F$ the sum $\sum_{i\in F}{k_i}$ (if $F=\emptyset$, then $\Sigma F=0$). Each of the finitely many vectors $T^{\Sigma F}x-x$ belongs to $E_{\aap}$ by Lemma \ref{lem:A1Am}.
%By Lemma \ref{lem:multap} we find a subsequence  $(n'_k)$ of $(n_k)$ such that
%\begin{equation*}
%\bigl\|T^{n'_{k}}(T^{\Sigma F}x-x)-T^{n'_{k-1}}(T^{\Sigma F}x-x)\bigr\|\leq \frac1{M2^m}\quad\mbox{for all $F\subseteq\bigl\{1,\dots,m\bigr\}$}
%\end{equation*}
%and with $n'_k-n'_{k-1}>k_m$ for all $k\geq K$, $K$ sufficiently large.
By using Lemma \ref{lem:multap} we find  $k,\ell\in\NN$ with $k-\ell>k_m$ such that
\begin{equation*}
\bigl\|T^{n_{k}}(T^{\Sigma F}x-x)-T^{n_{\ell}}(T^{\Sigma F}x-x)\bigr\|\leq \frac1{M2^m}\quad\mbox{for all $F\subseteq\bigl\{1,\dots,m\bigr\}$}.
\end{equation*}
By setting $k_{m+1}:=n_k-n_{\ell}$  we  obtain
\begin{equation}\label{eq:sigmaest2}
\bigl\|T^{{k_{m+1}}}P(T^{\Sigma F}x-x)-P(T^{\Sigma F}x-x)\bigr\|\leq \frac1{2^m}
\end{equation}
for all $F\subseteq\{1,\dots,m\}$.
We also have
$$
M\bigl\|P(T^{{k_{m+1}}}x-x)\bigr\|\geq \bigl\|T^{n_{\ell}}P(T^{{k_{m+1}}}x-x)\bigr\|=\bigl\|P(T^{n_{k}}x-T^{{n_{\ell}}}x)\bigr\|\geq \delta,
$$
and hence we obtain
$$
\bigl\|P(T^{{k_{m+1}}}x-x)\bigr\| \geq\frac{\delta}{M}.
$$

\medskip
Let $x_i:=P(T^{{k_i}}x-x)$. We claim that the sequence $(x_i)$ fulfills the conditions of Theorem \ref{thm:bespelcz}. Indeed, we have $\|x_i\|\geq \delta/M$ by construction so the series $\sum x_i$ cannot be convergent. For $m\in \NN$  and $1\leq i_1<\ i_2<\cdots<i_m$ we have
\begin{align*}
-\sum_{j=1}^m x_{i_j}
&=\hskip0.8em{T^{{k_{i_m}}}}P\Bigl(T^{{k_{i_1}}+\cdots+{k_{i_{m-1}}}}x-x\Bigr)-P\Bigl(T^{{k_{i_1}}+\cdots+{k_{i_{m-1}}}}x-x\Bigr)\\
&\quad+{T^{{k_{i_{m-1}}}}}P\Bigl(T^{{k_{i_1}}+\cdots+{k_{i_{m-2}}}}x-x\Bigr)-P\Bigl(T^{{k_{i_1}}+\cdots+{k_{i_{m-2}}}}x-x\Bigr)\\
&\quad\hskip0.5em\vdots\\
&\quad+{T^{{k_{i_2}}}}P\Bigl(T^{{k_{i_1}}}x-x\Bigr)-P\Bigl(T^{{k_{i_1}}}-x\Bigr)\\[2ex]
&\quad+P(x-T^{{k_{i_1}}+\cdots+{k_{i_m}}}x).
\end{align*}
By \eqref{eq:sigmaest2} we obtain
$$
\Bigl\|\sum_{j=1}^m x_{i_j}\Bigr\|\leq \sum_{j=2}^{m} \frac{1}{2^{i_j-1}}+M\|x\|+M^2\|x\|\leq M'<+\infty.
$$
It follows that $E$ contains a copy of $\co$, a contradiction.
\end{proof}

If $T$ is doubly power-bounded, then by Remark \ref{rem:dbly} we have $E_\st=\{0\}$, and hence $(I-P)=0$. So we obtain the following special case of Basit's more general result:
\begin{corollary}[Basit]\label{cor:basit}
Let $E$ be Banach space $E$ which does not contain a copy of $\co$, and let $T\in \LLL(E)$ be a doubly power-bounded operator. If $x\in E$ and $(T-I)x$ is an a.a.p.~vector, then so is $x$ itself.
\end{corollary}

The above results are certainly not valid for arbitrary Banach spaces.
A counterexample is actually provided by the very same one
showing that the analogue of the Bohl--Bohr theorem fails for arbitrary Banach-valued
functions, see \cite{KadecMI} or \cite[Sec.~2.1]{RuSu2}.
\begin{example}\label{exa:1}
Consider $E=\BUC(\RR;\co)$, and $T$ the shift by $a>0$, and
 $x(t):=(\sin \frac{t}{2^n})_{n\in \NN}$. Then $T\in \LLL(E)$ is doubly power-bounded, and we have
$$
\Bigl|\sin\tfrac{t+h}{2^n}-\sin \tfrac{t}{2^n}\Bigr|=\Bigl|\sin\tfrac{h}{2^{n+1}}\cos\tfrac{2t+h}{2^{n+1}}\Bigr|\leq \varepsilon
$$
for all $n\in\NN$ and $t\in \RR$ if $|h|$ is sufficiently small,
therefore $x\in \BUC(\RR;\co)$. On the other hand $x$ is not an a.a.p.~vector since the set
$$
\bigl\{(\sin\tfrac{t+ma}{2^n})_{n\in \NN}:m\in \NN,t\in \RR\bigr\}\subseteq \co
$$
is not  relatively compact. On the other hand,
 $$y(t):=[(T-I)x](t)=(\sin\tfrac{t+a}{2^n}-\sin\tfrac{t}{2^n})_{n\in \NN}=(\sin\tfrac{a}{2^{n+1}}\cos\tfrac{2t+a}{2^{n+1}})_{n\in
 \NN}$$ is almost periodic, because $y(t)_n\to 0$ uniformly in $t\in \RR$ as $n\to \infty$.
\end{example}

Next we show that on $\co$ itself there is bounded linear operator satisfying the assumptions of  Theorem \ref{thm:bolisnoco1} but for which the conclusion of that theorem fails to hold.
\begin{example}\label{exa:nobolisifco}
It suffices to exhibit an example on $E:=\cc$. Let $(a_n)\in\cc$ be a sequence with $|a_n|=1$ for all $n\in\NN$ and $a_n\neq  b:=\lim_{n\to\infty}a_n$  for all $n\in\NN$. Define
$$
T(x_n):=(a_nx_n).
$$
Then $T\in \LLL(E)$ with $\|T\|=1$. Moreover, $T$ is invertible and doubly power-bounded. Since the standard basis vectors of $\co$ are eigenvectors of $T$ corresponding to unimodular eigenvalues, they all belong to $E_\rv$, and hence $\co\subseteq E_\rv\subseteq E_\aap$. Moreover, since $T$ is doubly power-bounded we have $E_\st=\{0\}$, $I-P=0$, and the condition ``$(\id-P)(T^nx-x)$ converges'' is trivially satisfied for every $x\in E$. Not all vectors are a.a.p.~with respect to $T$. It suffices to prove this for the case when $b=\lim_{n\to\infty}a_n=1$,  otherwise we can pass to the operator $b^{-1}T$, which has precisely the same a.a.p.~vectors as $T$. Now suppose by contradiction that $E=E_\aap$ holds. Then $T$ is mean ergodic on $E$, which is equivalent to the fact that $\ker (T-\id)$ separates $\ker(T'-\id)$, see, e.g., \cite[Ch.~8]{EFHN}. But this is false, as $\dim \ker (T-\id)=0$ and $\dim\ker (T'-\id)\geq 1$. Hence $E\neq E_\aap$ and it also follows that $E_\rv=E_\aap=\co$.

\medskip\noindent Finally, we indeed suppose $b=1$. Then, since $\ran(T-\id)\subseteq \co=E_\aap$, we obtain that $(T-\id)x$ is a.a.p., for every $x\in E$, but not all $x\in E$ belongs to $E_\aap$.
\end{example}

By \cite[Sec.~2.5]{AlbiacKalton} if $\co$ is a closed subspace in a separable Banach space, then it is complemented in there. Thus Example \ref{exa:nobolisifco} in combination with Theorem \ref{thm:bolisnoco1} yields the following:
\begin{corollary}\label{cormain}
For a separable Banach space $E$  the following assertions are equivalent:
\begin{enumerate}[(i)]
\item The Banach space $E$ does no contain a copy of $\co$.
\item For every power-bounded linear operator $T\in \LLL(E)$ and $x\in E$ the orbit
$$
\{T^nx:n\in \NN\}\subseteq E
$$
is relatively compact if and only if
$$
\{T^{n+1}x-T^nx:n\in \NN\}\subseteq E
$$
is relatively compact and $(I-P)(T^nx-x)$ is convergent for $n\to\infty$.
\end{enumerate}
\end{corollary}

%
% \begin{theorem}
%If $E$ does not contain a copy of $\co$ then for every $x\in E$
%the set
% $$
% \bigl\{T^{n+1}x-T^nx:n\in \NN\bigr\}
% $$
% is relatively compact if and only if
%  $$
% \bigl\{T^{n+m}x-T^n x:n\in \NN\bigr\}
% $$
% is compact for all/some $m\in \NN$.
%\end{theorem}

We close this paper by two consequences of the previous results, interesting in their own right:
 \begin{corollary}\label{cor:A1Am}
 Let $E$ be Banach space not containing $\co$. Then for every  $x\in E$, $T\in \LLL(E)$ doubly power-bounded operator and $m\in \NN$, $m\geq 1$, the relative compactness of the two sets
 $$
 D_1:=\bigl\{T^{n+1}x-T^nx:n\in \NN\bigr\}\subseteq E
 $$
and
  $$
 D_m :=\bigl\{T^{n+m}x-T^n x:n\in \NN\bigr\}\subseteq E
 $$
 are equivalent.
\end{corollary}
\begin{proof}
If $D_m$ is relatively compact, then so is $\{T^{nm+m}x-T^{nm}x:n\in \NN\}$ and by Corollary \ref{cor:basit} even $\{T^{nm}x:n\in \NN\}$. By the continuity of $T$ the set $B_k:=\{T^{nm+k}x:n\in \NN\}$ is relatively compact for all $k=0,\dots, m-1$. Since
$$
\bigl\{T^{n}x:n\in \NN\bigr\}=B_1\cup B_2\cup\dots\cup B_{m-1},
$$
the relative compactness of the (forward) orbit of $x$ follows. But this implies the relative compactness of $D_1$.
That the relative compactness of $D_1$ implies that of $D_m$, is true without any assumption on the Banach space $E$, see Lemma \ref{lem:A1Am}.
\end{proof}

\begin{example}\label{exa:A1Am}
Let $E:=\cc$ and for $m\in\NN$, $m\geq 2$ fixed let $T$ be as in Example \ref{exa:1} with $\lim_{n\to\infty}a_n=b\in \CC$ an $m^{\text{th}}$ root of unity. Then we have $E_\aap=E_\rv=\co$. Since  $\rg{(T^{m}-I)}\subseteq \co$ and $\rg{(T-I)}\subseteq \co+(b-1)\one$ ($\one$ is the constant $1$ sequence), we obtain that for this doubly power-bounded operator $T$ and for every $x\in E$ the set $D_m$ as in Corollary \ref{cor:A1Am} is compact, while $D_1$ is not.
\end{example}

Similarly as for Corollary \ref{cormain}, we obtain from Corollary \ref{cor:A1Am} and Example \ref{exa:A1Am} the next characterization.
\begin{corollary}
A separable Banach space $E$ does not contain a copy of $\co$ if and only if  for every  $x\in E$, $T\in \LLL(E)$ doubly power-bounded operator and $m\in \NN$ the compactness of the two sets
 $$
\bigl\{T^{n+1}x-T^nx:n\in \NN\bigr\}
 $$
and
  $$
\bigl\{T^{n+m}x-T^n x:n\in \NN\bigr\}
 $$
 are equivalent.
\end{corollary}

\def\cprime{$'$}
\providecommand{\bysame}{\leavevmode\hbox to3em{\hrulefill}\thinspace}
\providecommand{\MR}{\relax\ifhmode\unskip\space\fi MR }
% \MRhref is called by the amsart/book/proc definition of \MR.
\providecommand{\MRhref}[2]{%
  \href{http://www.ams.org/mathscinet-getitem?mr=#1}{#2}
}
\providecommand{\href}[2]{#2}

\parindent0pt
\end{document}